\documentclass[12pt]{amsart}

\usepackage{amsmath,amsthm,amsfonts,amssymb,latexsym,amscd,bbm}
\usepackage{color}

\input{prepictex.tex}
\input{pictex.tex}
\input{postpictex.tex}

\begin{document}

\newtheorem{theorem}{Theorem}[section]
\newtheorem{corollary}[theorem]{Corollary}
\newtheorem{lemma}[theorem]{Lemma}
\newtheorem{proposition}[theorem]{Proposition}
\newtheorem{conjecture}[theorem]{Conjecture}
\newtheorem{commento}[theorem]{Comment}
\newtheorem{problem}[theorem]{Problem}
\newtheorem{remarks}[theorem]{Remarks}
\newtheorem{remark}[theorem]{Remark}
\newtheorem{example}[theorem]{Example}

\theoremstyle{definition}
\newtheorem{definition}[theorem]{Definition}

\newcommand{\Nb}{{\mathbb{N}}}
\newcommand{\Rb}{{\mathbb{R}}}
\newcommand{\Tb}{{\mathbb{T}}}
\newcommand{\Zb}{{\mathbb{Z}}}
\newcommand{\Cb}{{\mathbb{C}}}

\newcommand{\Af}{\mathfrak A}
\newcommand{\Bf}{\mathfrak B}
\newcommand{\Ef}{\mathfrak E}
\newcommand{\Gf}{\mathfrak G}
\newcommand{\Hf}{\mathfrak H}
\newcommand{\Kf}{\mathfrak K}
\newcommand{\Lf}{\mathfrak L}
\newcommand{\Mf}{\mathfrak M}
\newcommand{\Rf}{\mathfrak R}

\newcommand{\x}{\mathfrak x}

\def\A{{\mathcal A}}
\def\B{{\mathcal B}}
\def\C{{\mathcal C}}
\def\D{{\mathcal D}}
\def\E{{\mathcal E}}
\def\F{{\mathcal F}}
\def\G{{\mathcal G}}
\def\H{{\mathcal H}}
\def\J{{\mathcal J}}
\def\K{{\mathcal K}}
\def\LL{{\mathcal L}}
\def\N{{\mathcal N}}
\def\M{{\mathcal M}}
\def\N{{\mathcal N}}
\def\OO{{\mathcal O}}
\def\P{{\mathcal P}}
\def\Q{{\mathcal Q}}
\def\SS{{\mathcal S}}
\def\T{{\mathcal T}}
\def\U{{\mathcal U}}
\def\W{{\mathcal W}}

\def\ext{\operatorname{Ext}}
\def\span{\operatorname{span}}
\def\clsp{\overline{\operatorname{span}}}
\def\Ad{\operatorname{Ad}}
\def\ad{\operatorname{Ad}}
\def\tr{\operatorname{tr}}
\def\id{\operatorname{id}}
\def\en{\operatorname{End}}
\def\aut{\operatorname{Aut}}
\def\out{\operatorname{Out}}
\def\coker{\operatorname{coker}}

\def\la{\langle}
\def\ra{\rangle}
\def\rh{\rightharpoonup}
\def\cl{\textcolor{blue}{$\clubsuit$}}

\title[On Conjugacy of Subalgebras]{On Conjugacy of Subalgebras in Graph $C^*$-Algebras. II}

\author{Tomohiro Hayashi, Jeong Hee Hong,  and Wojciech Szyma{\'n}ski}



\thanks{W. Szyma\'{n}ski was supported by the DFF-Reesearch Project 2 on `Automorphisms and invariants of operator algebras', 
Nr. 7014--00145B.}

\date{\today}

\begin{abstract} 
We apply a method inspired by Popa's intertwining-by-bimodules technique to investigate inner conjugacy of MASAs in graph $C^*$-algebras. First we give 
a new proof of non-inner conjugacy of the diagonal MASA $\D_E$ to its non-trivial image under a quasi-free automorphism, where $E$ is a finite 
transitive graph. Changing graphs representing the algebras, this result applies to some non quasi-free automorphisms as well. 
Then we exhibit a large class of MASAs in the Cuntz algebra $\OO_n$ that are not inner conjugate to the diagonal $\D_n$. 
\end{abstract}

\maketitle



\section{Introduction}

This paper is devoted to investigations of conjugacy of MASAs in the $C^*$-algebras of finite directed graphs. 
The problem of conjugacy of MASAs in factor von Neumann algebras has been extensively investigated for many years, in particular with relation 
to Cartan subalgebras. Variety of different situations may occur. There exist factors with a unique Cartan subalgebra or with (uncountably) many 
ones, e.g. see \cite{Packer, SpV, HoVa}. 

This problem has received much less attention by researchers working with $C^*$-algebras. In particular, the literature 
on conjugacy of subalgebras in simple purely infinite $C^*$-algebras is rather scarce. The present paper is continuation of investigations of this problem 
initiated in \cite{CHS2015} and \cite{HHMSz}, where the question of inner conjugacy to the diagonal MASA of its images under quasi-free automorphisms 
was looked at in the Cuntz algebras and more generally graph $C^*$-algebras. The arguments from \cite{CHS2015} and \cite{HHMSz} where based 
on rather ad hoc estimations, tailor made for the cases at hand. Now, we aim at developing a more general technique that may be applicable in 
many diverse instances. The idea is simple, see Lemma \ref{calgibbm} below, and it is inspired by Popa's intertwining-by-bimodules technique, 
see Theorem \ref{intbybimod} below. We believe that this approach is conceptually sound and may be useful in many a different situation. 

Our paper is organized as follows. 
Section 2 contains rather extensive preliminaries on graph $C^*$-algebras, traces on them, and their endomorphisms. 
In particular, a discussion of aspects of the classical Perron-Frobenius theory is included, in so far as it is relevant for our purpose. At the end of this section, 
we briefly state the key technical device we intend to use for distinguishing non-inner conjugate subalgebras. 
Section 3 contains a discussion of quasi-free automorphisms in relation to aspects of the Perron-Frobenius theory. In this section we give a new, 
and hopefully conceptually more interesting, proof of non-inner conjugacy to the diagonal of its images under non-trivial quasi-free automorphisms, 
see Theorem \ref{qfree} below. 
In section 4, we show that by changing graph representing the algebra in question our main result on quasi-free automorphism becomes 
applicable to some non quasi-free automorphisms as well. 
In section 5, we exhibit a large class of MASAs of the Cuntz algebra $\OO_n$ that are not inner conjugate to the diagonal MASA $\D_n$, 
thus generalizing the case resulting from quasi-free automorphisms. 
In the final section 6, we collected proofs of a few technical lemmas needed in the preceding parts of the paper.


\section{Preliminaries}\label{sectionpreliminaries}

\subsection{Finite directed graphs and their $C^*$-algebras}

Let $E=(E^0,E^1,r,s)$ be a directed graph, where $E^0$ and $E^1$ are {\em finite} sets of vertices 
and edges, respectively, and $r,s:E^1\to E^0$ are range and source maps, respectively. 
A {\em path} $\mu$ of length $|\mu|=k\geq 1$ is a sequence 
$\mu=(\mu_1,\ldots,\mu_k)$ of $k$ edges $\mu_j$ such that 
$r(\mu_j)=s(\mu_{j+1})$ for $j=1,\ldots, k-1$. We view the vertices as paths of length $0$.
The set of all paths of length $k$ is denoted $E^k$, and $E^*$ denotes the collection of 
all finite paths (including paths of length zero). The range and source maps naturally 
extend from edges $E^1$ to paths $E^k$. A {\em sink} is a vertex $v$ which emits 
no edges, i.e. $s^{-1}(v)=\emptyset$. A {\em source} is a vertex $w$ which receives 
no edges, i.e. $r^{-1}(w)=\emptyset$. By a {\em cycle} we mean a path $\mu$ of 
length $|\mu|\geq 1$ such that $s(\mu)=r(\mu)$. 
A cycle $\mu=(\mu_1,\ldots,\mu_k)$ has an exit if there is a $j$ such that $s(\mu_j)$ 
emits at least two distinct edges. If $\alpha$ is an initial subpath of $\beta$ then we write $\alpha\prec\beta$. 
Graph $E$ is {\em transitive} if for any two vertices $v,w$ there exists 
a path $\mu\in E^*$ from $v$ to $w$ of non-zero length. Thus a transitive graph does not contain any 
sinks or sources. Given a graph $E$, we will denote by $A=[A(v,w)]_{v,w\in E^0}$ its {\em adjacency matrix}.  
That is, $A$ is a matrix with rows and columns indexed by the vertices of $E$, such that $A(v,w)$ is 
the number of edges with source $v$ and range $w$. If the graph $E$ is transitive then the corresponding matrix $A$ is {\em irreducible}, in the sense 
that for any two vertices $v,w$ there is a positive integer $k$ such that $A^k(v,w)>0$. Here $A^k$ is the $k$'th power of matrix $A$ and 
hence $A^k(w, v)$ gives the number of paths from vertex $w$ to vertex $v$. 

The $C^*$-algebra $C^*(E)$ corresponding to a graph $E$ 
is by definition, \cite{KPRR} and \cite{KPR},  the universal $C^*$-algebra generated by mutually 
orthogonal projections $P_v$, $v\in E^0$, and partial isometries $S_e$, $e\in E^1$, 
subject to the following two relations: 
\begin{itemize}
\item[(GA1)] $S_e^*S_e=P_{r(e)}$,  
\item[(GA2)] $P_v=\sum_{s(e)=v}S_e S_e^*$ if $v\in E^0$ emits at least one edge. 
\end{itemize}
For a path $\mu=(\mu_1,\ldots,\mu_k)$ we denote by $S_\mu=
S_{\mu_1}\cdots S_{\mu_k}$ the corresponding partial isometry in $C^*(E)$.  
 We agree to write $S_v=P_v$ for a  $v\in E^0$.
Each $S_\mu$ is non-zero with the domain projection $P_{r(\mu)}$. 
Then $C^*(E)$ is the closed span of $\{S_\mu S_\nu^*:\mu,\nu\in E^*\}$.  
Note that $S_\mu S_\nu^*$ is non-zero if and only if 
$r(\mu)=r(\nu)$. In that case, $S_\mu S_\nu^*$ is a partial isometry with domain and range projections 
equal to $S_\nu S_\nu^*$ and $S_\mu S_\mu^*$, respectively. 

The range projections $P_\mu=S_\mu S_\mu^*$ of all 
partial isometries $S_\mu$ mutually commute, and the abelian $C^*$-subalgebra of $C^*(E)$ 
generated by all of them is called the diagonal subalgebra and denoted $\D_E$. 
We set $\D^0_E = {\rm span}\{P_v  :  v\in E^0  \}$ and, more generally, 
$\D_E^k= {\rm span}\{P_\mu  :  \mu\in E^k  \}$ for $k\geq 0$. $C^*$-algebra $\D_E$ coincides with the 
norm closure of $\bigcup_{k=0}^\infty\D_E^k$. 
If $E$ does not contain sinks and all cycles have exits then 
$\D_E$ is a MASA (maximal abelian subalgebra) in $C^*(E)$ by \cite[Theorem 5.2]{HPP}. 
Throughout this paper, we make the following

\vspace{2mm}\noindent
{\bf standing assumption:}  all graphs we consider are finite, transitive  
and all cycles in these graphs admit exits. 

\vspace{2mm}
There exists a strongly continuous action $\gamma$ of the circle group $U(1)$ on $C^*(E)$, 
called the {\em gauge action}, such that $\gamma_z(S_e)=zS_e$ and $\gamma_z(P_v)=P_v$ 
for all $e\in E^1$, $v\in E^0$ and $z\in U(1)\subseteq\Cb$. 
The fixed-point algebra $C^*(E)^\gamma$ for the gauge action is an AF-algebra, denoted 
$\F_E$ and called the core AF-subalgebra of $C^*(E)$. $\F_E$ is the closed span of 
$\{S_\mu S_\nu^*:\mu,\nu\in E^*,\;|\mu|=|\nu|\}$. For $k\in\Nb=\{0,1,2,\ldots\}$ 
we denote by $\F_E^k$ the linear span of $\{S_\mu S_\nu^*:\mu,\nu\in E^*,\;|\mu|=|\nu|= k\}$. 
$C^*$-algebra $\F_E$ coincides with the norm closure of $\bigcup_{k=0}^\infty\F_E^k$. 

We consider the usual {\em shift} on $C^*(E)$, \cite{CK}, given by
\begin{equation}\label{shift}
\varphi(x)=\sum_{e\in E^1} S_e x S_e^*, \;\;\; x\in C^*(E).  
\end{equation} 
In general, for finite graphs without sinks and sources, the shift is a unital, completely positive map. However, it 
is an injective $*$-homomorphism when restricted to the relative commutant
$(\D_E^0)'\cap C^*(E)$  of $\D_E^0$ in $C^*(E)$. 

We observe that for each $v\in E^0$ projection $\varphi^k(P_v)$ is minimal in the center of $\F_E^k$. 
The $C^*$-algebra $\F_E^k\varphi^k(P_v)$ is the linear span of partial isometries $S_\mu S_\nu^*$ with 
$|\mu|=|\nu|=k$ and $r(\mu)=r(\nu)=v$. It is isomorphic to the full matrix algebra of size $\sum_{w\in E^0}
A^k(w, v)$. The multiplicity of $\F_E^k\varphi^k(P_v)$ in $\F_E^{k+1}\varphi^{k+1}(P_w)$ is $A(v,w)$, 
so the Bratteli diagram for $\F_E$ is induced from the graph $E$, see \cite{CK}, \cite{KPRR} or \cite{BPRSz}.   
\[ \beginpicture

\setcoordinatesystem units <1.5cm,1.2cm>
\setplotarea x from -4 to 5, y from -2 to 2
\put {$\bullet$} at -2 1
\put {$\bullet$} at -1 1
\put {$\bullet$} at 2 1
\put {$\bullet$} at -2 -1
\put {$\bullet$} at 1 -1
\put {$\bullet$} at 2 -1
\put {$\hdots$} at 1 1
\put {$\hdots$} at -1 -1
\setlinear
\plot -0.9 0.9  0.9 -0.9 /
\arrow <0.235cm> [0.2,0.6] from 0.7 -0.7 to 0.9 -0.9
\put {$\F_E\varphi^k(P_v)$} at -1 1.4
\put {$\F_E\varphi^{k+1}(P_w)$} at 1 -1.4
\put {$\F_E^k$} at 3 1
\put {$\F_E^{k+1}$} at 3 -1
\put {$A(v,w)$} at 0.5 0.1

\endpicture \] 
We denote 
\begin{equation}
 \Bf := (\D_E^0)'\cap \F_E^1. 
\end{equation}
That is, $\Bf$ is the linear span of elements $S_e S_f^*$, $e,f\in E^1$, with $s(e)=s(f)$. We note that 
$\Bf$ is contained in the multiplicative domain of $\varphi$. We have $\D_E^1 \subseteq \Bf \subseteq \F_E^1$ and 
\begin{equation}\label{fekrelcomm}
\varphi^k(\Bf) = (\F_E^k)' \cap \F_E^{k+1} \cong \bigoplus_{v,w\in E^0}M_{A(v,w)}(\Cb)
\end{equation}
for all $k$. For $v,w\in E^0$, we denote 
\begin{equation}\label{vQw}
_vQ_w := \sum_{e\in E^1, s(e)=v, r(e)=w}P_e. 
\end{equation}
Each $_vQ_w$ is a minimal projection in the center of $\Bf$ and $\Bf_vQ_w\cong M_{A(v,w)}(\Cb)$. We put
\begin{equation}\label{Bk}
\Bf_E^k := \bigvee_{j=0}^{k-1}\varphi^j(\Bf), 
\end{equation}
for $k\geq 1$, 
the $C^*$-algebra generated by $\bigcup_{j=0}^{k-1}\varphi^j(\Bf)$. In general, 
if $A$ and $B$ are both $C^*$-subalgebras of a $C^*$-algebra $C$, then we denote by $A\vee B$ the $C^*$-subalgebra 
of $C$ generated by $A$ and $B$. 
Since for all $k$ we have 
\begin{equation}
\D_E^k =  \bigvee_{j=0}^{k-1}\varphi^j(\D_E^1), 
\end{equation}
it is easy to see that
\begin{equation}\label{dbf}
\D_E^k \subseteq \Bf_E^k \subseteq \F_E^k. 
\end{equation}
We observe that 
\begin{equation}\label{prQ}
_vQ_w\varphi(_{v'}Q_{w'}) = \delta_{w,v'} \sum_{s(e)=v, r(e)=s(f)=w, r(f)=w'} P_{ef}. 
\end{equation}
This implies that 
$$ \begin{aligned}
\Bf_E^k & = \bigoplus_{v_1,\ldots,v_{k+1}\in E^0} \Bf _{v_1}Q_{v_2} \vee \varphi(\Bf _{v_2}Q_{v_3}) \vee \ldots \vee 
\varphi^{k-1}(\Bf_{v_k}Q_{v_{k+1}}) \\
 & = \bigoplus_{v_1,\ldots,v_{k+1}\in E^0} \Bf _{v_1}Q_{v_2} \otimes \varphi(\Bf _{v_2}Q_{v_3}) \otimes \ldots \otimes 
\varphi^{k-1}(\Bf_{v_k}Q_{v_{k+1}}). 
\end{aligned} $$

There exist faithful conditional expectations $\Phi_{\F}:C^*(E)\to\F_E$ and $\Phi_{\D}:C^*(E)\to\D_E$ 
such that $\Phi_{\F}(S_\mu S_\nu^*)=0$ for $|\mu|\neq|\nu|$ and $\Phi_{\D}(S_\mu S_\nu^*)=0$ 
for $\mu\neq\nu$. We note that $\Phi_{\D} = \Phi_{\D}\circ \Phi_{\F}$ and 
$$ \begin{aligned}
\Phi_{\D}\circ\varphi & = \varphi\circ\Phi_{\D} \;\;\; \text{on } \D_E, \\
\Phi_{\F}\circ\varphi & = \varphi\circ\Phi_{\F} \;\;\; \text{on } \F_E.
\end{aligned} $$
For an integer $m\in\Zb$, we denote by $C^*(E)^{(m)}$ the spectral subspace of the gauge 
action corresponding to $m$. That is,  
\begin{equation}\label{msp} 
C^*(E)^{(m)}:=\{x\in C^*(E) \mid \gamma_z(x)=z^m x,\,\forall z\in U(1)\}. 
\end{equation} 
In particular, $C^*(E)^{(0)}=C^*(E)^\gamma$. For each 
$m\in\Zb$ there is a unital, contractive and completely bounded map $\Phi^m:C^*(E)\to C^*(E)^{(m)}$ given by 
\begin{equation}
\Phi^m(x) = \int_{z\in U(1)} z^{-m}\gamma_z(x)dx. 
\end{equation}
In particular, $\Phi^0=\Phi_\F$.
We have $\Phi^m(x)=x$ for all $x\in C^*(E)^{(m)}$.  If 
$x\in C^*(E)$ and $\Phi^m(x)=0$ for all $m\in\Zb$ then $x=0$. 


\subsection{The trace on the core AF-subalgebra}

We recall the definition of a natural trace on the core $AF$-subalgebra $\F_E$. For relevant facts from the 
Perron-Frobenius theory, see for example \cite{GHJ}, \cite{G}. 

Let $\beta$ be the Perron-Frobenius eigenvalue of the matrix $A$ and let $(x(v))_{v\in E^0}$ be the corresponding 
Perron-Frobenius eigenvector. That is, $\beta>0$, for each $v\in E^0$ we have $x(v)>0$, and 
\begin{equation}\label{pfev}
\sum_{w\in E^0} A(v,w) x(w) = \beta x(v). 
\end{equation}
We set $X:=\sum_{v\in E^0} x(v)$ and define a tracial state $\tau$ on $\F_E$ so that 
\begin{equation}\label{trace}
\tau(S_{\mu}S_{\nu}^*) = \delta_{\mu,\nu}\frac{x(r(\mu))}{X\beta^k} 
\end{equation}
for $\mu,\nu\in E^k$. We have $\tau(\Phi_{\D}(x)) = \tau(x)$ for all $x\in\F_E$.  

\begin{remark}\label{traceshift}\rm 
Trace $\tau$ defined above is not shift invariant, in general. That is, it may happen that $\tau(\varphi(x))\neq\tau(x)$ for some $x\in\F_E$. In fact, 
$\tau$ is $\varphi$-invariant if and only if 
$$ \sum_{v\in E^0} A(v,w) = \beta $$ 
for each $w\in E^0$. For example, the matrix 
$$ A = \left( \begin{array}{cc} 2 & 1 \\ 1 & 4 \end{array} \right) $$
does not satisfy this condition. 
\end{remark}


\subsection{Endomorphisms determined by unitaries}

Cuntz's classical approach to the study of endomorphisms of $\OO_n$, \cite{Cun}, has  been developed further in \cite{CSz} and 
extended to graph $C^*$-algebras in \cite{CHS2012}, \cite{AJSz} and \cite{JSSz}. 

We denote by $\U_E$ the collection of all those unitaries in $C^*(E)$ which 
commute with all vertex projections $P_v$, $v\in E^0$. That is 
\begin{equation}\label{ue}
\U_E:=\U((\D_E^0)'\cap C^*(E)). 
\end{equation}
If $u\in\U_E$ then $uS_e$, $e\in E^1$, are partial isometries in $C^*(E)$ which together with 
projections $P_v$, $v\in E^0$, satisfy (GA1) and (GA2). Thus, by the universality of $C^*(E)$, 
there exists a unital $*$-homomorphism $\lambda_u:C^*(E)\to C^*(E)$ such 
that\footnote{The reader should be aware that in some papers (e.g. in \cite{Cun}) a 
different convention is used, namely $\lambda_u(S_e)=u^* S_e$.}
\begin{equation}\label{lambda}
\lambda_u(S_e)=u S_e \;\;\; {\rm and}\;\;\;  \lambda_u(P_v)=P_v, \;\;\;  {\rm for}\;\; e\in E^1,\; v\in E^0. 
\end{equation}
The mapping $u\mapsto\lambda_u$ establishes 
a bijective correspondence between $\U_E$ and the semigroup of those unital endomorphisms 
of $C^*(E)$ which fix all  $P_v$, $v\in E^0$.
As observed in \cite[Proposition 2.1]{CHS2012}, if $u\in\U_E\cap\F_E$ then $\lambda_u$ 
is automatically injective. We say $\lambda_u$ is {\em invertible} if $\lambda_u$ is an automorphism of $C^*(E)$. 
If $u$ belongs to $\U_E\cap\F_E^k$ for some $k$, then the corresponding endomorphism $\lambda_u$ is called 
{\em localized}, \cite{CP}, \cite{CHS2012}. 

 If $u\in\U(\Bf)$ then $\lambda_u$ is automatically invertible with inverse 
$\lambda_{u^*}$ and the map 
\begin{equation}\label{quasifree}
\U(\Bf) \ni u\mapsto \lambda_u \in\aut(C^*(E))
\end{equation} 
is a group homomorphism with range inside 
the subgroup of {\em quasi-free automorphisms} of $C^*(E)$, see \cite{Z}. Note that this group is almost never trivial 
and it is non-commutative if graph $E$ contains two edges $e,f\in E^1$ such that $s(e)=s(f)$ and $r(e)=r(f)$. 

The shift $\varphi$ globally preserves $\U_E$, $\F_E$ and $\D_E$. For $k\geq 1$ we denote 
\begin{equation}\label{uk}
u_k := u\varphi(u)\cdots\varphi^{k-1}(u).   
\end{equation}
For each $u\in\U_E$ and all $e\in E^1$ we have $S_e u = \varphi(u) S_e$, and thus 
\begin{equation}\label{uaction}
\lambda_u(S_\mu S_\nu^*)=u_{|\mu|}S_\mu S_\nu^*u_{|\nu|}^* 
\end{equation}
for any two paths $\mu,\nu\in E^*$. 


\subsection{The Popa criterion}\label{sectionacriterion}

In the analysis of uniqueness of Cartan subalgebras of tracial von Neumann algebras, 
Popa's {\em intertwining-by-bimodules} technique has been extremely successful. This method goes back to \cite{Popa}, 
but has been polished over the years and recently even extended to type $III$ case, \cite{HoIs}. 
 The following result contains its essential ingredient. 

\begin{theorem}[S. Popa]\label{intbybimod}
Let $M$ be a von Neumann algebra equipped with a faithful normal trace $\tau$. Let $A,B$ be von Neumann subalgebras 
of $M$, and let $\Phi_B:M\to B$ be a $\tau$-preserving conditional expectation. Then the following two conditions are equivalent. 
\begin{itemize}
\item[(1)] There exist non-zero projections $p\in A$, $q\in B$, a non-zero partial isometry $v\in pMq$ and a $*$-homomorphism 
$\phi:pAp\to qBq$ such that $xv=v\phi(x)$ for all $x\in pAp$. 
\item[(2)] There is no sequence of unitaries $w_n\in\U(A)$ such that 
\begin{equation}\label{zerolimit}
|| \Phi_B(xw_n y)||_2 \underset{n\to\infty}{\longrightarrow} 0, \;\;\; \forall x,y\in M. 
\end{equation}
\end{itemize}
\end{theorem}

This beautiful theorem is inapplicable to graph $C^*$-algebras, of course. However, the following simple fact 
remains valid in the $C^*$-algebraic setting. 

\begin{lemma}\label{calgibbm}
Let $M$ be a unital $C^*$-algebra, and let $A,B$ be its $C^*$-subalgebras containing the unit of $M$. Let $\Phi_B:M\to B$ be 
a conditional expectation, and let $\tau$ be a trace on $B$. If there is a sequence of unitaries 
$w_n\in\U(A)$ such that (\ref{zerolimit}) holds then there is no unitary $v\in\U(M)$ such that $vAv^*\subseteq B$. 
\end{lemma}
\begin{proof}
Indeed, let $w_n\in\U(A)$ be as in the lemma and suppose $v\in\U(M)$ is such that $vAv^*\subseteq B$. Then 
$$ 1=||vw_nv^*||_2=|| \Phi_B(vw_n v^*)||_2 \underset{n\to\infty}{\longrightarrow} 0,  $$
which gives a contradiction. 
\end{proof}


\section{Quasi-free automorphisms}

In this section, we apply Lemma \ref{calgibbm} with $M=C^*(E)$, $\tau$ the canonical trace on $\F_E$, $B=\D_E$, and $\Phi_B=\Phi_\D$. 
We keep the standing assumptions on the graph $E$. 
Note that for unitaries $u\in\Bf$ and $d\in\D_E^1$ we have 
$$
\lambda_u(d\varphi(d)\cdots \varphi^{k-1}(d)) = udu^*\varphi(udu^*)\cdots \varphi^{k-1}(udu^*). 
$$

\begin{lemma}\label{dlemma}
Let $u\in\Bf$ be a unitary such that $u\D_E^1u^*\neq \D_E^1$, and let $d\in\D_E^1$ be a unitary such that $udu^*\not\in\D_E^1$. 
Then we have 
$$
\lim_{k\to\infty}|| \Phi_{\D}(udu^*\varphi(udu^*)\cdots \varphi^{k-1}(udu^*)) ||_2 = 0. 
$$
\end{lemma}
\begin{proof}
We set $d_{v,w}:=d\cdot_vQ_w$. Since $\Bf\cdot_vQ_w$ is a full matrix algebra, it has a unique tracial state 
$\tau_{v,w}$. We denote by 
$||\cdot||_{2,v,w}$ the 2-norm induced by this trace. In view of Corollary \ref{ce}, below, we have 
$$ \begin{aligned}
\Phi_\D(udu^* & \varphi(udu^*)\cdots \varphi^{k-1}(udu^*)) \\
& = \sum_{v_1,v_2,\ldots,v_{k+1}\in E^0} \Phi_\D(udu^*_{v_1}Q_{v_2}\varphi(udu^*_{v_2}Q_{v_3})\cdots \varphi^{k-1}(udu^*_{v_k}Q_{v_{k+1}})) \\
& = \sum_{v_1,v_2,\ldots,v_{k+1}\in E^0} \Phi_\D(ud_{v_1,v_2}u^*) \varphi(\Phi_\D( ud_{v_2,v_3}u^*)) \cdots \varphi^{k-1}(\Phi_\D(ud_{v_k,v_{k+1}}u^*))
\end{aligned} $$
We define non-negative numbers $\{ \lambda_{v_1,v_2,\ldots,v_{k+1}} \}_{v_1,v_2,\ldots,v_{k+1}\in E^0}$ by 
$$ \begin{aligned}
 \lambda_{v_1,v_2,\ldots,v_{k+1}} & = \tau(_{v_1}Q_{v_2} \varphi(_{v_2}Q_{v_3}) \cdots \varphi^{k-1}(_{v_k}Q_{v_{k+1}})) \\
 & = A(v_1,v_2) A(v_2,v_3) \cdots A(v_k,v_{k+1}) \frac{x(v_{k+1})}{X\beta^k}.
\end{aligned} $$ 
We remark that $A(v_1,v_2) A(v_2,v_3) \cdots A(v_k,v_{k+1})$ is the total number of paths  of length $k$ which pass through $v_1,v_2,\ldots,v_{k+1}$ 
in this order. 
Since $_{v_1}Q_{v_2} \varphi(_{v_2}Q_{v_3}) \cdots \varphi^{k-1}(_{v_k}Q_{v_{k+1}})$ is a central minimal 
projection of $\Bf_E^k$, for any $x\in\Bf_E^k$, we have 
$$
\tau(x)=\sum_{v_1,v_2,\ldots,v_{k+1}\in E^0} \lambda_{v_1,v_2,\ldots,v_{k+1}} 
\tau_{{v_1,v_2,\ldots,v_{k+1}}}(x\{_{v_1}Q_{v_2} \varphi(_{v_2}Q_{v_3}) \cdots \varphi^{k-1}(_{v_k}Q_{v_{k+1}})\})
$$
where $\tau_{{v_1,v_2,\ldots,v_{k+1}}}$ is a unique tracial state on a full matrix algebra 
$${\Bf_E^k} \{{}_{v_1}Q_{v_2} \varphi(_{v_2}Q_{v_3}) \cdots\varphi^{k-1}(_{v_k}Q_{v_{k+1}})\}.$$ 
Then since 
$$ \begin{aligned}
\tau_{{v_1,v_2,\ldots,v_{k+1}}}&( a_1\varphi(a_2)\cdots\varphi^{k-1}(a_k)
\{{}_{v_1}Q_{v_2} \varphi(_{v_2}Q_{v_3}) \cdots \varphi^{k-1}(_{v_k}Q_{v_{k+1}})\}
)\\
&=\tau_{v_1,v_2}(a_1 \cdot _{v_1}Q_{v_2})\tau_{v_2,v_3}(a_2 \cdot _{v_2}Q_{v_3})\cdots 
\tau_{v_k,v_{k+1}}(a_k \cdot _{v_k}Q_{v_{k+1}})
\end{aligned} 
$$ 
for all $a_1,a_2,\ldots,a_k\in\Bf$, we have  
$$ \begin{aligned}
&  || a_1\varphi(a_2)\cdots\varphi^{k-1}(a_k) ||_2^2\\
=\sum_{v_1,v_2,\ldots,v_{k+1}\in E^0}  \lambda_{v_1,v_2,\ldots,v_{k+1}} &  || a_1 \cdot _{v_1}Q_{v_2} ||^2_{2,v_1,v_2} 
|| a_2 \cdot _{v_2}Q_{v_3} ||^2_{2,v_2,v_3} \cdots || a_k \cdot _{v_k}Q_{v_{k+1}} ||^2_{2,v_k,v_{k+1}}.
\end{aligned} $$ 
Thus we see that 
$$ \begin{aligned}
 || \Phi_\D(udu^* \varphi(udu^*)&\cdots \varphi^{k-1}(udu^*)) ||_2^2 
= || \Phi_\D(udu^*) \varphi(\Phi_\D(udu^*))\cdots \varphi^{k-1}(\Phi_\D(udu^*)) ||_2^2 \\
 = \sum_{v_1,v_2,\ldots,v_{k+1}\in E^0} &  \lambda_{v_1,v_2,\ldots,v_{k+1}} 
||\Phi_\D(ud_{v_1,v_2}u^*)||^2_{2,v_1,v_2} ||\varphi(\Phi_\D( ud_{v_2,v_3}u^*))||^2_{2,v_2,v_3} \cdots \\ 
& \cdots ||\varphi^{k-1}(\Phi_\D(ud_{v_k,v_{k+1}}u^*)) ||^2_{2,v_k,v_{k+1}}.
\end{aligned} $$

By the hypothesis of the lemma, there exist two vertices $w_1,w_2$ such that 
$$
\begin{aligned}
 &0<||udu^*\cdot _{w_1}Q_{w_2}- \Phi_\D( udu^*\cdot _{w_1}Q_{w_2}) ||^2_{2,w_1,w_2}\\
&=||udu^*\cdot _{w_1}Q_{w_2}||^2_{2,w_1,w_2}+||\Phi_\D( udu^*\cdot _{w_1}Q_{w_2}) ||^2_{2,w_1,w_2}
-2{\rm Re}\tau_{w_1,w_2}(\{udu^*\cdot _{w_1}Q_{w_2}\}^*\Phi_\D( udu^*\cdot _{w_1}Q_{w_2}))\\
&=1+||\Phi_\D( udu^*\cdot _{w_1}Q_{w_2}) ||^2_{2,w_1,w_2}
-2{\rm Re}\tau_{w_1,w_2}(\Phi_\D( udu^*\cdot _{w_1}Q_{w_2})^*\Phi_\D( udu^*\cdot _{w_1}Q_{w_2}))\\
&=1- || \Phi_\D( udu^*\cdot _{w_1}Q_{w_2}) ||^2_{2,w_1,w_2}
\end{aligned}
$$ 
and hence 
\begin{equation}\label{cconstant}
c:= || \Phi_\D( udu^*\cdot _{w_1}Q_{w_2}) ||^2_{2,w_1,w_2} < 1.
\end{equation}
For $i=0,1,\ldots,k$, we denote by $M_{k,v}^i$ the set of all paths $\mu$ such that 

\vspace{2mm}\noindent(i)
$|\mu|=k$, 

\vspace{2mm}\noindent(ii)
$r(\mu)=v$,

\vspace{2mm}\noindent(iii)
in path $\mu$, edges from $w_1$ to $w_2$ occur exactly $i$ times. 

\vspace{2mm}\noindent
We remark that $M_{k,v}^i \cap M_{k,v}^j=\emptyset$ if $i\neq j$. Thus we have $\sum_{i=0}^k|M_{k,v}^i| = \sum_{w\in E^0}A^k(w,v)$, 
where $|M_{k,v}^i|$ denotes the cardinality of $M_{k,v}^i$. We claim that for all $v$ and $i$
\begin{equation}\label{mlemma}
\lim_{k\to\infty} \frac{|M_{k,v}^i|}{\beta^k} =0.
\end{equation}
At first we note that because of (\ref{cconstant}) the full matrix algebra $\Bf\cdot _{w_1}Q_{w_2}$ is not isomorphic to $\Cb$, and hence $A(w_1,w_2)\geq 2$. 
Let $A_1$ be the matrix defined in (\ref{a1}) in section 6 below for $(i_1,j_1)=(w_1,w_2)$, and let $E_1$ be the corresponding graph.  $E_1$ may be viewed 
as a subgraph of $E$ obtained by removing all but one edge in $E^1$ that begin at $w_1$ and end at $w_2$. Set $N_{k,v}^i := M_{k,v}^i\cap E_1^*$. 
It is easy to see that 
$$
|M_{k,v}^i| = |N_{k,v}^i|\cdot A(w_1,w_2)^i. 
$$
But now, by virtue of Theorem \ref{pf3} below, we have 
$$
\frac{|M_{k,v}^i|}{\beta^k} = A(w_1,w_2)^i \cdot \frac{|N_{k,v}^i|}{\beta^k} \leq  A(w_1,w_2)^i \cdot \frac{\sum_{w}A_1^k(v,w)}{\beta^k} 
\underset{k\to\infty}{\longrightarrow} 0, 
$$
and the claim holds. 

Now, since $||\varphi^{j-1}(\Phi_{\mathcal D}(ud_{v_j,v_{j+1}}u^*)) ||^2_{2,v_j,v_{j+1}}\leq 1$ and 
$c= || \Phi_{\mathcal D}( udu^*\cdot _{w_1}Q_{w_2}) ||^2_{2,w_1,w_2}$, for each $i_0$ we have 
$$ \begin{aligned}
 ||  \Phi_\D( & udu^*  \varphi(udu^*)\cdots \varphi^{k-1}(udu^*)) ||_2^2 \leq \sum_{v\in E^0} \frac{x(v)}{X\beta^k} \sum_{i=0}^k |M_{k,v}^i| c^i \\ 
& =   \sum_{v\in E^0} \frac{x(v)}{X\beta^k} \sum_{i=0}^{i_0} |M_{k,v}^i| c^i + \sum_{v\in E^0} \frac{x(v)}{X\beta^k} \sum_{i=i_0+1}^{k}|M_{k,v}^i|c^i, 
\end{aligned} $$
and hence 
$$
\limsup_{k\to\infty} ||  \Phi_\D(udu^*  \varphi(udu^*)\cdots \varphi^{k-1}(udu^*)) ||_2^2 = \limsup_{k\to\infty} 
\sum_{v\in E^0} \frac{x(v)}{X\beta^k} \sum_{i=i_0+1}^{k}|M_{k,v}^i|c^i. 
$$
Since 
$$ \begin{aligned}
\sum_{v\in E^0} \frac{x(v)}{X\beta^k} & \sum_{i=i_0+1}^{k}|M_{k,v}^i|c^i = 
c^{i_0}\sum_{v\in E^0} \frac{x(v)}{X\beta^k} \sum_{i=i_0+1}^{k}|M_{k,v}^i|c^{i-i_0} \\
\leq c^{i_0} \sum_{v\in E^0} & \frac{x(v)}{X\beta^k}  \sum_{i=i_0+1}^{k}|M_{k,v}^i| \leq 
c^{i_0} \sum_{v\in E^0} \frac{x(v)}{X\beta^k} \sum_{w\in E^0}A^{k}(w,v) \\ 
& = c^{i_0} \frac{1}{X\beta^k} \sum_{w\in E^0}  \sum_{v\in E^0}A^{k}(w,v)x(v) = c^{i_0},  
\end{aligned} $$
we may conclude that 
$$
\limsup_{k\to\infty} ||  \Phi_\D(udu^*  \varphi(udu^*)\cdots \varphi^{k-1}(udu^*)) ||_2^2 \leq c^{i_0}. 
$$
Since $i_0$ was arbitrary, the lemma is proved. 
\end{proof}

Keeping the hypothesis of Lemma \ref{dlemma}, we have the following.

\begin{lemma}\label{flemma}
For all $x,y\in\F_E$ we have 
$$
\lim_{k\to\infty} ||  \Phi_\D(x\cdot udu^*  \varphi(udu^*)\cdots \varphi^{k-1}(udu^*) \cdot y) ||_2 = 0. 
$$
\end{lemma}
\begin{proof}
To prove the lemma, it suffices to consider elements $x,y\in\F_E^p$ for an arbitrary positive integer $p$. We have 
$$ \begin{aligned}
 & x\cdot udu^*  \varphi(udu^*)\cdots \varphi^{k-1}(udu^*) \cdot y \\
 & = (x\cdot udu^*  \varphi(udu^*)\cdots \varphi^{p-1}(udu^*) \cdot y) \cdot 
\varphi^p(udu^* \varphi^{1}(udu^*)\cdots \varphi^{k-1}(udu^*)). 
\end{aligned} $$
Therefore it is enough to show that 
$$
\lim_{k\to\infty} ||  \Phi_\D(x\cdot \varphi^p(udu^*  \varphi(udu^*)\cdots \varphi^{k-1}(udu^*) )) ||_2 = 0 
$$
for all $x\in\F_E^p$. However, we have 
$$
\Phi_\D(x\cdot \varphi^p(udu^*  \varphi(udu^*)\cdots \varphi^{k-1}(udu^*) )) = 
\Phi_\D(x)\cdot \varphi^p(\Phi_\D(udu^*  \varphi(udu^*)\cdots \varphi^{k-1}(udu^*) )) 
$$
by Lemma \ref{productcondexp} below and 
$$
\lim_{k\to\infty} || \varphi^p(\Phi_\D(udu^*  \varphi(udu^*)\cdots \varphi^{k-1}(udu^*) )) ||_2 = 0
$$
by Lemma \ref{dlemma} and Lemma \ref{newlemma}, below. Thus the claim follows. 
\end{proof}

Now we are ready to prove the main result of this section. We keep the standard assumptions on the graph $E$.

\begin{theorem}\label{qfree}
Let $u\in\Bf$ be a unitary such that $u\D_E^1u^*\neq \D_E^1$, and let $d\in\D_E^1$ be a unitary such that $udu^*\not\in\D_E^1$. 
Then for all $x,y\in C^*(E)$ we have 
$$
\lim_{k\to\infty}|| \Phi_{\D}(x\cdot udu^*\varphi(udu^*)\cdots \varphi^{k-1}(udu^*) \cdot y) ||_2 = 0. 
$$
Thus, in view of Lemma \ref{calgibbm}, $\D_E$ and $\lambda_u(\D_E)$ are not inner conjugate in $C^*(E)$.
\end{theorem}
\begin{proof}
By the polarization identity, it suffices to compute the above limit in the case $y=x^*$. Furthermore, we may assume that $x$ belongs to the dense 
$*$-subalgebra of $C^*(E)$ generated by partial isometries corresponding to finite paths. That is, in the case $x$ is a finite sum of the form 
$$
x = \sum_{\mu\in E^*} a_\mu S_\mu^* + x_0    + \sum_{\nu\in E^*} S_\nu b_\nu,
$$
with $x_0,a_\mu,b_\nu\in\F_E$. Applying conditional expectation $\Phi_\F$ on the core AF-subalgebra first, we get
$$ \begin{aligned}
 \Phi_\F(x\cdot udu^*\varphi(udu^*)\cdots \varphi^{k-1}(udu^*) \cdot x^*) 
 = & \sum_{|\mu|=|\mu'|} a_\mu S_\mu^* \cdot udu^*\varphi(udu^*)\cdots \varphi^{k-1}(udu^*) \cdot S_{\mu'}a_{\mu'}^* \\
+ &  x_0 \cdot udu^*\varphi(udu^*)\cdots \varphi^{k-1}(udu^*) \cdot x_0^* \\
+ &  \sum_{|\nu|=|\nu'|} S_\nu b_\nu \cdot udu^*\varphi(udu^*)\cdots \varphi^{k-1}(udu^*) \cdot b_{\nu'}^* S_{\nu'}^*.
\end{aligned} $$
Thus we must show the following three cases:
$$ \begin{aligned}
(1) \;\;& \lim_{k\to\infty}|| \Phi_{\D}( \sum_{|\mu|=|\mu'|} a_\mu S_\mu^* 
\cdot udu^*\varphi(udu^*)\cdots \varphi^{k-1}(udu^*) \cdot S_{\mu'}a_{\mu'}^* )||_2 = 0, \;\;\;\;\;\;\;\;  \;\;\;\;\;\;\;\;\;\;\;\;\;\;\;\;\;\;\;\;\;\;\;\; \\ 
(2) \;\;& \lim_{k\to\infty}|| \Phi_{\D}( x_0 \cdot udu^*\varphi(udu^*)\cdots \varphi^{k-1}(udu^*) \cdot x_0^* )||_2 = 0, \\
(3) \;\; & \lim_{k\to\infty}|| \Phi_{\D}( \sum_{|\nu|=|\nu'|} S_\nu b_\nu \cdot udu^*\varphi(udu^*)\cdots 
\varphi^{k-1}(udu^*) \cdot b_{\nu'}^* S_{\nu'}^* )||_2 = 0.
\end{aligned} $$
Ad (1). Consider two paths $\mu$ and $\mu'$ with $|\mu|=|\mu'|$. 
For any $x\in\Bf$ and for any $l-1\geq|\mu|=|\mu'|$, we see that 
$$
\varphi^{l-1}(x)S_{\mu'}=\sum_{|\nu|=l-1}S_\nu xS_\nu^*S_{\mu'}
=\sum_{|\nu'|=l-1-|\mu'|}S_{\mu'}S_{\nu'} xS_{\nu'}^*S_{\mu'}^*S_{\mu'}
=S_{\mu'}\varphi^{l-1-|\mu|}(x).
$$
On the other hand, for any $l-1<|\mu|=|\mu'|$, 
since both $S_\mu S_\mu^*$ and $S_{\mu'} S_{\mu'}^*$ are minimal projections 
of $\Bf_E^{|\mu|}$ and 
$x\varphi(x)\cdots\varphi^{l-1}(x)\in \Bf_E^{|\mu|}$, 
we have 
$$ 
\begin{aligned}
S_{\mu}^*x\varphi(x)\cdots\varphi^{l-1}(x)S_{\mu'}&=
S_{\mu}^*(S_\mu S_\mu^*)x\varphi(x)\cdots\varphi^{l-1}(x)(S_{\mu'} S_{\mu'}^*)S_{\mu'}\\
&=S_{\mu}^*(tS_\mu S_{\mu'}^*)S_{\mu'}
=t\delta_{r(\mu),r(\mu')}
P_{r(\mu)}
\end{aligned}
$$
for some scalar $t\in\Cb$ with $|t|\leq ||x||^{l-1}$. 
Therefore for any $k>|\mu|=|\mu'|$ we see that 
$$
S_\mu^* 
\cdot udu^*\varphi(udu^*)\cdots \varphi^{k-1}(udu^*) \cdot S_{\mu'} = 
t\delta_{r(\mu),r(\mu')}
P_{r(\mu)}\cdot udu^*\varphi(udu^*)\cdots \varphi^{k-1-|\mu|}(udu^*)
$$ 
for some scalar  $t\in\Cb$ with $|t|\leq ||udu^*||^{k-1}=1$. 
Since $P_{r(\mu)}\in \D$, the claim follows from Lemma \ref{dlemma}.

\vspace{3mm}\noindent
Ad (2). This is shown in Lemma \ref{flemma}. 

\vspace{3mm}\noindent
Ad (3). If $\nu\neq\nu'$ then 
$$
\Phi_{\D}( S_\nu b_\nu \cdot udu^*\varphi(udu^*)\cdots \varphi^{k-1}(udu^*) \cdot b_{\nu'}^* S_{\nu'}^* ) =0. 
$$
Thus 
$$ \begin{aligned}
 & || \Phi_{\D}( \sum_{|\nu|=|\nu'|} S_\nu b_\nu \cdot udu^*\varphi(udu^*)\cdots 
\varphi^{k-1}(udu^*) \cdot b_{\nu'}^* S_{\nu'}^* )||_2 \\
 & \leq  \sum_{|\nu|=|\nu'|} || \Phi_{\D}( \varphi^{|\nu|}( b_\nu \cdot udu^*\varphi(udu^*)\cdots \varphi^{k-1}(udu^*) \cdot b_{\nu'}^* ) )||_2,
\end{aligned} $$
and this tends to $0$ as $k$ increases to infinity by the same argument as in the proof of Lemma \ref{flemma}. 
\end{proof}


\section{An application --- changing graphs}

The same graph $C^*$-algebra may often be presented by many different graphs, and the property of being quasi-free 
is usually not preserved when passing from one graph to another. This makes Theorem \ref{qfree} applicable 
to a much wider class of automorphisms than quasi-free ones. We illustrate this phenomenon with the following two examples. 

\subsection{Out-splitting}

Given a graph $E$ satisfying our standing assumption, we consider its out-split graph $E_s(\P)$, as defined by Bates and Pask in \cite{BP}. Namely, 
for each vertex $v\in E^0$ we partition the set of edges emitted by $v$, that is $s^{-1}(v)$, 
into $m(v)$ non-empty disjoint subsets $E^1_v, \ldots, E^{m(v)}_v$. 
Denote by $\P$ the resulting partition of $E^1$. The out-split graph $E_s(\P)$ has the following vertices, edges, source and range functions:
$$ \begin{aligned}
E_s(\P)^0 & = \{ v^i \mid v\in E^0, \, 1\leq i \leq m(v)  \}, \\
E_s(\P)^1 & = \{ e^j \mid e\in E^1, \, 1\leq j \leq m(r(e))  \}, \\
s(e^j) & = s(e)^i, \:\:\: \text{and} \:\:\: r(e^j)  = r(e)^j.
\end{aligned} $$
As shown in \cite[Theorem 3.2]{BP}, the $C^*$--algebras $C^*(E)$ and $C^*(E_s(\P))$ are isomorphic by an isomorphism which maps the diagonal 
MASA $\D_E$ of $C^*(E)$ onto the diagonal MASA $\D_{E_s(\P)}$ of $C^*(E_s(\P))$. However, the groups of quasi-free 
automorphims of $C^*(E)$ and $C^*(E_s(\P))$ may be different. For example, in the following case 
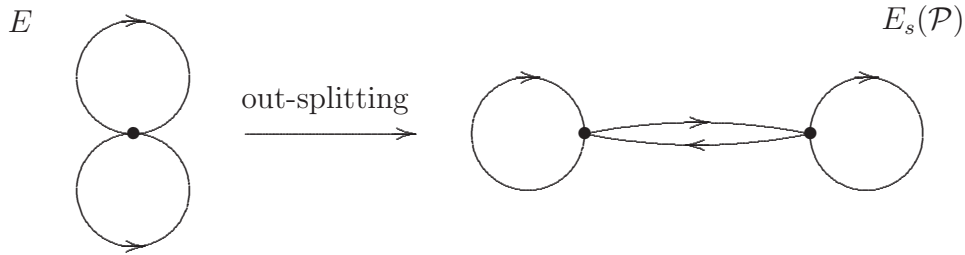
\[ \beginpicture
\setcoordinatesystem units <1.5cm,1.5cm>
\setplotarea x from -6 to 6, y from -1 to 1
\put {$\bullet$} at -4 0
\put {$\bullet$} at 0 0
\put {$\bullet$} at 2 0 
\put {$E$} at -5 1
\put {$E_s(\P)$} at 3 1
\put {out-splitting} at -2.3 0.3
\setlinear
\plot -3 0  -1.5 0 /
\arrow <0.25cm> [0.2,0.6] from -1.7 0 to -1.5 0
\setquadratic
\plot 0 0   1 0.1  2 0 /
\plot 0 0   1 -0.1  2 0 /
\circulararc 360 degrees from 2 0 center at 2.5 0
\circulararc 360 degrees from 0 0 center at -0.5 0
\circulararc 360 degrees from -4 1 center at -4 0.5
\circulararc 360 degrees from -4 -1 center at -4 -0.5
\arrow <0.25cm> [0.2,0.6] from 0.9 0.1 to 1.1 0.1
\arrow <0.25cm> [0.2,0.6] from 1.1 -0.1 to 0.9 -0.1
\arrow <0.25cm> [0.2,0.6] from -4.1 1 to -3.9 1
\arrow <0.25cm> [0.2,0.6] from -4.1 -1 to -3.9 -1
\arrow <0.25cm> [0.2,0.6] from 2.4 0.5 to 2.6 0.5
\arrow <0.25cm> [0.2,0.6] from -0.6 0.5 to -0.4 0.5
\endpicture \] 
the groups $\U(\Bf)$ in $C^*(E_s(\P))$ and $C^*(E)$ are isomorphic to $U(1)\times U(1)\times U(1)\times U(1)$ 
and $U(2)$, respectively. 

Thus the isomorphism $C^*(E)\cong C^*(E_s(\P))$ may identify a quasi-free 
automorphism of one of the two algebras with a non quasi-free automorphism of the other. In this way, our Theorem \ref{qfree} leads to 
non-trivial examples of non quasi-free automorphisms of graph algebras that map the diagonal MASA onto another MASA that is not inner conjugate to it. 


\subsection{Two graphs for $\OO_2$}

Consider the following graph:
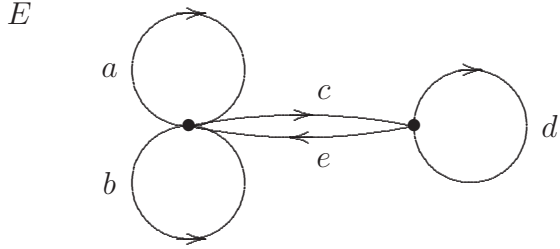
\[ \beginpicture
\setcoordinatesystem units <1.5cm,1.5cm>
\setplotarea x from -2 to 3.5, y from -1.3 to 1.3
\put {$\bullet$} at 0 0
\put {$\bullet$} at 2 0 
\put {$E$} at -1.5 1
\put {$a$} at -0.7 0.5
\put {$b$} at -0.7 -0.5
\put {$c$} at 1.2 0.3
\put {$d$} at 3.2 0
\put {$e$} at 1.2 -0.3
\setquadratic
\plot 0 0   1 0.1  2 0 /
\plot 0 0   1 -0.1  2 0 /
\circulararc 360 degrees from 2 0 center at 2.5 0
\circulararc 360 degrees from 0 1 center at 0 0.5
\circulararc 360 degrees from 0 -1 center at 0 -0.5
\arrow <0.25cm> [0.2,0.6] from 0.9 0.1 to 1.1 0.1
\arrow <0.25cm> [0.2,0.6] from 1.1 -0.1 to 0.9 -0.1
\arrow <0.25cm> [0.2,0.6] from -0.1 1 to 0.15 1
\arrow <0.25cm> [0.2,0.6] from -0.1 -1 to 0.15 -1
\arrow <0.25cm> [0.2,0.6] from 2.4 0.5 to 2.6 0.5
\endpicture \] 
Then the graph algebra $C^*(E)$  is isomorphic to the Cuntz algebra $\OO_2=C^*(T_1,T_2)$, \cite{Cun77}. Here  $C^*(T_1,T_2)$ is the universal 
$C^*$-algebra for the relations $ 1 = T_1T_1^* + T_2T_2^* = T_1^*T_1 = T_2^*T_2$. That is, it is a graph algebra for the graph consisting of one 
vertex and two edges attached to it. The isomorphism between 
$C^*(E) = C^*(S_a,S_b,S_c,S_d,S_e)$ and $\OO_2=C^*(T_1,T_2)$ is obtained by the identification 
$$ S_a=T_{11}T_1^*, \;\;S_b=T_{121}T_1^*,\;\; S_c=T_{122}T_2^*,\;\; S_d=T_{22}T_2^*,\;\; S_e=T_{21}T_1^*. $$ 
The inverse map is given by 
$$ T_1=S_a + (S_b +S_c)(S_d + S_e)^*,\;\; T_2=S_d + S_e. $$ 
Note that this isomorphism carries the diagonal MASA $\D_E$ of $C^*(E)$ onto the standard diagonal MASA $\D_2$ of $\OO_2$. 
Indeed, it follows from the above definition that every product $x$ of the generators of $C^*(E)$ is mapped onto an element of the form $T_\alpha T_\beta^*$ 
in $\OO_2$, with $\alpha,\beta$ some words on the alphabet $\{1,2\}$. Thus the range projection $xx^*$ of that product is mapped onto a projection 
in the diagonal MASA $\D_2$ of $\OO_2$. Hence, the image of $\D_E$ is contained in $\D_2$. But under an 
isomorphism, the image of a MASA in $C^*(E)$ is a MASA in $\OO_2$. Thus the image of MASA $\D_E$ of $C^*(E)$ is the entire MASA $\D_2$ 
of $\OO_2$, as claimed.

Now, using the isomorphism above we may find a quasi-free automorphism of $C^*(E)$ such that the corresponding automorphism of $\OO_2$ 
is not quasi-free and yet carries the diagonal MASA $\D_2$ of $\OO_2$ onto a MASA which is not inner conjugate to $\D_2$. Indeed, let 
$$ \left[ \begin{array}{cc} \xi_{aa} & \xi_{ab} \\ \xi_{ba} & \xi_{bb} \end{array} \right] $$
be a unitary matrix whose all entries are non-zero complex numbers. Then 
$$ u=\xi_{aa}S_a S_a^* + \xi_{ab}S_a S_b^* + \xi_{ba}S_b S_a^* + \xi_{bb}S_b S_b^* + S_c S_c^* + S_d S_d^* + S_e S_e^* $$ 
is a unitary in $\F_E^1$ satisfying the hypothesis of Theorem \ref{qfree}. The above isomorphism transports the quasi-free automorphism $\lambda_u$ 
of $C^*(E)$ onto an automorphism $\lambda_U$ of $\OO_2$ corresponding to the unitary
$$ U= \xi_{aa}T_{11}T_{11}^* + \xi_{ab}T_{11}T_{121}^* + \xi_{ba}T_{121}T_{11}^* + \xi_{bb}T_{121}T_{121}^* + 
T_{122}T_{122}^* + T_2T_2^*. $$
Clearly, $\lambda_U$ is not a quasi-free automorphism of $\OO_2$, since $U$ does not belong to the linear span of $T_iT_j^*$, $i,j=1,2$. 
Theorem \ref{qfree} implies that  there is no unitary $w\in\OO_2$ satisfying $w\D_2w^*=\lambda_U(\D_2)$. 


\section{Certain MASAs in $\OO_n$ not inner conjugate to the diagonal $\D_n$}

In this section, we consider the Cuntz algebra $\OO_n$, with $2\leq n<\infty$. As usual, we view it as the graph $C^*$-algebra 
of the graph $E_n$ with one vertex and $n$ edges. Let $\lambda_u\in\en(\OO_n)$. Suppose 
$w_k$ is a sequence of unitaries in a commutative $C^*$-subalgebra $A$ of $\OO_n$. 
We ask under what circumstances the sequence $w_k$ satisfies the condition
of Lemma \ref{calgibbm} for $M=\OO_n$, $A$, $B=\D_n$, and $\tau$ the canonical trace on the UHF-subalgebra 
$\F_n$. Clearly, this is the case if and only if 
\begin{equation}\label{czerolimit}
|| \Phi_{\D_n}(S_\alpha S_\beta^* w_k S_\mu S_\nu^*) ||_2 \underset{k\to\infty}{\longrightarrow} 0, 
\end{equation}
for all paths $\alpha,\beta,\mu,\nu$. Let 
\begin{equation}\label{onfs}
w_k = \sum_{m\in\Zb}w_k^{(m)}
\end{equation}
be the standard Fourier series of $w_k$ (with respect to the decomposition of $\OO_n$ into spectral subspaces $\OO_n^{(m)}$ for the gauge action). Then 
(\ref{czerolimit}) is equivalent to the requirement that 
\begin{equation}\label{cfszerolimit}
|| \Phi_{\D_n}(S_\alpha S_\beta^* w_k^{(m)} S_\mu S_\nu^*) ||_2 \underset{k\to\infty}{\longrightarrow} 0, 
\end{equation}
for all paths $\alpha,\beta,\mu,\nu$, and all $m\in\Zb$. Of course, it suffices to consider the case $m=|\beta|+|\nu|-|\alpha|-|\mu|$. 
Clearly, for all $x\in\OO_n$ and all paths $\alpha$ we have 
\begin{equation}\label{shiftingnorm}
|| \Phi_{\D_n}(S_\alpha xS_\alpha^*)||_2 = n^{-|\alpha|/2}|| \Phi_{\D_n}(x)||_2. 
\end{equation}
Thus it suffices to consider condition (\ref{cfszerolimit}) in the following three cases:

\vspace{3mm}\noindent
(ZL1) $\nu=\emptyset$, $\beta\neq\emptyset$ and $m=|\beta| - |\alpha| - |\mu|$,

\vspace{3mm}\noindent
(ZL2)  $\alpha=\emptyset$, $\mu\neq\emptyset$ and   $m=|\beta|+|\nu|-|\mu|$, 

\vspace{3mm}\noindent
(ZL3)  $\alpha=\nu=\emptyset$ and   $m=|\beta|-|\mu|$. 

\begin{lemma}\label{zl3king}
If (\ref{cfszerolimit}) holds for all $\alpha$, $\beta$, $\mu$, $\nu$ as in (ZL3), it holds for all $\alpha$, $\beta$, $\mu$, $\nu$ as in (ZL1) and (ZL2).
\end{lemma}
\begin{proof}
Consider condition (ZL1) first. By (ZL3), we have 
$$
|| \Phi_{\D_n}(S_\beta^* w_k^{(m)} S_\mu S_\alpha) ||_2 \underset{k\to\infty}{\longrightarrow} 0.
$$
Thus by identity (\ref{shiftingnorm}) we also have 
$$
|| \Phi_{\D_n}(S_\alpha S_\beta^* w_k^{(m)} S_\mu ) ||_2 = || \Phi_{\D_n}(S_\alpha S_\alpha^* S_\alpha S_\beta^* w_k^{(m)} S_\mu ) ||_2 = 
|| \Phi_{\D_n}(S_\alpha S_\beta^* w_k^{(m)} S_\mu S_\alpha S_\alpha^*) ||_2 
$$
$$ = n^{-|\alpha|/2} || \Phi_{\D_n}(S_\beta^* w_k^{(m)} S_\mu S_\alpha) ||_2 \underset{k\to\infty}{\longrightarrow} 0.
$$
Now, consider condition (ZL2). By (ZL3), we have 
$$
|| \Phi_{\D_n}(S_\nu^* S_\beta^* w_k^{(m)} S_\mu ) ||_2 \underset{k\to\infty}{\longrightarrow} 0.
$$
Thus by identity (\ref{shiftingnorm}) we also have 
$$
|| \Phi_{\D_n}(S_\beta^* w_k^{(m)} S_\mu S_\nu^* ) ||_2 = || \Phi_{\D_n}(S_\beta^* w_k^{(m)} S_\mu S_\nu^*S_\nu S_\nu^* ) ||_2 = 
|| \Phi_{\D_n}(S_\nu S_\nu^* S_\beta^* w_k^{(m)} S_\mu S_\nu^*) ||_2 
$$
$$
= n^{-|\nu|/2} || \Phi_{\D_n}(S_\nu^* S_\beta^* w_k^{(m)} S_\mu ) ||_2 \underset{k\to\infty}{\longrightarrow} 0.
$$
\end{proof}
Now, we describe a construction of a large family of MASAs of the Cuntz algebra $\OO_n$ which are contained in the core UHF-subalgebra 
$\F_n$ and are not inner conjugate to the diagonal MASA $\D_n$. MASAs obtained by applying to $\D_n$ quasi-free automorphisms not 
preserving $\D_n$ provide very special examples of this more general construction. 

We start with a sequence $\{r_k\}_{k=1}^\infty$ of positive integers, and denote $R_1:=0$ and $R_k:=\sum_{j=1}^{k-1}r_j$ for $k\geq 2$. 
For each $k$ pick a  $0<c_k<1$ so that 
$$
\prod_{k=1}^\infty c_k =0.
$$
For each $k$ let $d_k$ be a unitary in $\varphi^{R_k}(\D_n^{r_k})$ and $U_k$ a unitary in $\varphi^{R_k}(\F_n^{r_k})$ such that 
\begin{equation}\label{dkUk}
|| \Phi_{\D_n}(U_k d_k U_k^*)||_2 \leq c_k. 
\end{equation}
Such unitaries may be found through easy matrix considerations. 
Given this data, we define $\A$ to be the $C^*$-subalgebra of $\OO_n$ generated by the union of all algebras $U_k \varphi^{R_k}(\D_n^{r_k}) U_k^*$. 

\begin{theorem}\label{newmasas}
Every $C^*$-algebra $\A$, defined as above, is a MASA in $\OO_n$ that is not inner conjugate to $\D_n$.
\end{theorem}
\begin{proof}
Let $\A$ be as above. Then clearly $\A$ is a MASA in the core UHF-subalgebra $\F_n$ of $\OO_n$, for example see \cite{DP}. 
We will show that  $\A$ is a MASA in the entire $\OO_n$ as well. 

Indeed, let $x$ be in $\A'\cap\OO_n$, and let $x = \sum_{m\in\Zb}x^{(m)}$ be 
its standard Fourier series. Then for each $m$ we have $x^{(m)}\in \A'\cap\OO_n$. Consider a fixed $m>0$. Both $x^{(m)*}x^{(m)}$ 
and $x^{(m)}x^{(m)*}$ are in $\A'\cap \F_n = \A$. Since both these elements are positive and for every projection $q\in\A$ we have 
$||qx^{(m)}x^{(m)*}|| = ||qx^{(m)}x^{(m)*}q|| = ||x^{(m)*}qx^{(m)}|| = ||qx^{(m)*}x^{(m)}||$, 
it easily follows that $x^{(m)*}x^{(m)} = x^{(m)}x^{(m)*}$. That is, the element 
$x^{(m)}$ of $\OO_n$ is normal. Now, denote $v=S_1^m$ and $a=x^{(m)}(S_1^m)^*$. Then $a\in\F_n$ and we have $x^{(m)}=av$.
Since $av$ is normal, we have 
$$
\tau(a^*a) = \tau(vv^*a^*a) = \tau(avv^*a^*) = \tau(v^*a^*av) = n^m \tau(a^*a). 
$$
Thus $\tau(a^*a)=0$ and hence $a=0$. Consequently, $x^{(m)}=0$ for all $m>0$. A similar argument shows that $x^{(m)} = 0$ 
for all $m<0$. Thus $x=x^{(0)}$ belongs to $\A'\cap\F_n = \A$, and $\A$ is a MASA in $\OO_n$ as claimed.

To show that $\A$ is not inner conjugate in $\OO_n$ to $\D_n$, we verify that condition 
(ZL3) holds for 
$$
w_k := \prod_{j=1}^k U_j d_j U_j^*. 
$$ 
Since each $w_k$ is in $\F_n$, it suffices to check it with $m=0$. So fix $\beta,\mu$ with $|\beta|=|\mu|$. Take $t$ so large that 
$t\geq|\beta|$ and consider $k> t$. Since $\prod_{j=t+1}^k U_j d_j U_j^*$ is in the range of injective endomorphism $\varphi^{|\mu|}$, we have
$$ \begin{aligned}
|| \Phi_{\D_n}(S_\beta^* w_k S_\mu ) ||_2 & = || \Phi_{\D_n}(S_\beta^* (\prod_{j=1}^{t} U_j d_j U_j^*) 
 \varphi^{|\mu|}(\varphi^{-|\mu|}(\prod_{j=t+1}^k U_j d_j U_j^*))S_\mu ) ||_2 \\
 & = || \Phi_{\D_n}(S_\beta^* (\prod_{j=1}^{t} U_j d_j U_j^*)S_{\mu} \varphi^{-|\mu|}(\prod_{j=t+1}^k U_j d_j U_j^*)) ||_2
\end{aligned} $$
Thus we have by Lemma \ref{productcondexp}, below, that 
$$ \begin{aligned}
|| \Phi_{\D_n}(S_\beta^* w_k S_\mu ) ||_2 & = || \Phi_{\D_n}(S_\beta^* (\prod_{j=1}^{t} U_j d_j U_j^*)S_{\mu}) 
\Phi_{\D_n}(\varphi^{-|\mu|}(\prod_{j=t+1}^k U_j d_j U_j^*)) ||_2 \\ 
 & \leq || \Phi_{\D_n}(S_\beta^* (\prod_{j=1}^{t} U_j d_j U_j^*)S_{\mu}) || \cdot 
|| \Phi_{\D_n}(\varphi^{-|\mu|}(\prod_{j=t+1}^k U_j d_j U_j^*)) ||_2 \\ 
 & = || \Phi_{\D_n}(S_\beta^* (\prod_{j=1}^{t} U_j d_j U_j^*)S_{\mu}) || \cdot || \prod_{j=t+1}^k \Phi_{\D_n}(U_j d_j U_j^*) ||_2 \\
 & = || \Phi_{\D_n}(S_\beta^* (\prod_{j=1}^{t} U_j d_j U_j^*)S_{\mu}) || \cdot \prod_{j=t+1}^k || \Phi_{\D_n}(U_j d_j U_j^*) ||_2  \\
 & \leq || \Phi_{\D_n}(S_\beta^* (\prod_{j=1}^{t} U_j d_j U_j^*)S_{\mu}) || \cdot \prod_{j=t+1}^k  c_k \underset{k\to\infty}{\longrightarrow} 0.
\end{aligned} $$
\end{proof}

We remark that it is not immediately clear which of the MASAs considered in Proposition \ref{newmasas} are outer conjugate in $\OO_n$ to $\D_n$.


\section{Technical lemmas}\label{sectiontechnicallemmas}

In this section, we collect a few technical facts used in the proofs above. 


\subsection{The conditional expectations}

\begin{lemma}\label{productcondexp}
Let $A$ and $B$ be unital $C^*$-subalgebras of a finite-dimensional $C^*$-algebra, such that $ab=ba$ for all $a\in A$, $b\in B$. 
Let $D_A$ and $D_B$ be MASAs of $A$ and $B$, respectively, so that $D:=D_A\vee D_B$ is a MASA of $A\vee B$. Let $\tau$ be a faithful  
tracial state on $A\vee B$, and let $E_D$, $E_{D_A}$ and $E_{D_B}$ be $\tau$-preserving conditional expectations from $A\vee B$ onto 
$D$, $D_A$ and $D_B$, respectively. Then we have 
$$ E_D(ab) = E_{D_A}(a)E_{D_B}(b) $$ 
for all $a\in A$, $b\in B$. 
\end{lemma}
\begin{proof}
If $A$ is a full matrix algebra (i.e., the center of $A$ is trivial) then $A\vee B\cong A\otimes B$ and $\tau(ab)=\tau(a)\tau(b)$ for all $a\in A$, $b\in B$. 
Thus, in this case, the claim obviously holds. 

In the general case, let $\{p_1,\ldots,p_n\}$ be the minimal central projections in $A$. Then 
$$ A\vee B = \bigoplus_{i=1}^n(A\vee B){p_i} \cong \bigoplus_{i=1}^n A{p_i}\otimes B{p_i}. $$
The $\tau$-preserving conditional expectation $E_i$ from $(A\vee B){p_i}$ onto $(D_A\vee D_B){p_i}$ satisfies 
$$ E_i(ab)=E_{D_A p_i}(a)E_{D_B p_i}(b) $$ 
for all $a\in Ap_i$ and $b\in Bp_i$, by the preceding argument. Since 
$$ E_D(x) = \sum_{i=1}^n E_i(xp_i), $$
the claim follows. 
\end{proof}

\begin{corollary}\label{ce}
For all $x_1,x_2,\ldots,x_k\in\Bf$ we have
$$
\Phi_{\D}(x_1\varphi(x_2)\cdots\varphi^{k-1}(x_k)) = \Phi_{\D}(x_1)\varphi(\Phi_{\D}(x_2))\cdots \varphi^{k-1}(\Phi_{\D}(x_k)).
$$
\end{corollary}
\begin{proof}
Since $\Bf$, $\varphi(\Bf)$, \ldots, $\varphi^{k-1}(\Bf)$ are mutually commuting 
unital finite-dimensional $C^*$-algebras, by Lemma 
\ref{productcondexp} we have 
$$
\Phi_{\D}(x_1\varphi(x_2)\cdots\varphi^{k-1}(x_k)) = \Phi_{\D}(x_1)\Phi_{\D}(\varphi(x_2))\cdots \Phi_{\D}(\varphi^{k-1}(x_k)).
$$
The claims follows since the conditional expectation $\Phi_{\D}$ commutes with the shift $\varphi$.
\end{proof}


\subsection{The Perron-Frobenius theory}

Let $A$ be an $n\times n$ matrix with non-negative integer entries. We assume that $A$ is {\em irreducible} in the sense that 
for each pair of indices $(i,j)$ there exists a positive integer $k$ such that $A^k(i,j)>0$. Let $\beta$ be the Perron-Frobenius eigenvalue 
and let $(x(1),x(2),\ldots,x(n))$ be the corresponding Perron-Frobenius eigenvector. That is, $\beta>0$, $x(i)>0$ for all indices $i=1,\ldots,n$, and 
$$
\sum_j A(i,j) x(j) = \beta x(i). 
$$
In this subsection, for a (not necessary square) matrix $B$ we write $B\geq 0$ if $B(i,j)\geq 0$ for all $(i,j)$. Likewise, we write 
$B> 0$ if $B(i,j)> 0$ for all $(i,j)$. For a column vector $y\geq 0$, we set 
$$
\lambda(y,A) = \max\{ \lambda\geq 0 \mid Ay\geq\lambda y \}.
$$
The following lemma is part of the classical Perron-Frobenius theory, hence its proof is omitted. 

\begin{lemma}\label{pf1}
For an irreducible matrix $A$, as above, we have 
$$
\beta = \max\{ \lambda(y,A) \mid y\geq 0, ||y||=1  \}.
$$
\end{lemma}

\begin{lemma}\label{newlemma}
Let $\beta'>0$ be the Perron-Frobenius eigenvalue of the 
transpose matrix ${^t}A$. Let $\{y(v)\}_v$ be the 
Perron-Frobenius eigenvector of ${^t}A$. That is, 
$$
\sum_v
A(v,w)y(v)=\beta' y(w).
$$
Set $m=\min_v y(v)$ and $M=\max_v y(v)$. 
For any $f\in {\mathcal D}_E$, we have 
$$
\tau(\varphi(f))\leq \dfrac{\beta' M}{\beta m}\tau(f).
$$
Hence we have 
$$
||\varphi^p(f)||_2^2\leq  \Bigl(\dfrac{\beta'M}
{\beta m}\Bigr)^p||f||_2^2.
$$
\end{lemma}
\begin{proof}
We may assume that $f=S_\mu S_\mu^*$. 
We see that 
$$ \begin{aligned}
\tau(\varphi(S_\mu S_\mu^*))&=
\sum_{e,r(e)=s(\mu)}\tau(S_{e\mu}S_{e\mu}^*)=
\sum_{v}A(v,s(\mu))
\dfrac{x(r(\mu))}{X\beta^{|\mu|+1}}\\
&\leq \sum_{v}A(v,s(\mu))\dfrac{y(v)}{m}
\dfrac{x(r(\mu))}{X\beta^{|\mu|+1}}=
\beta'\dfrac{y(s(\mu))}{m}
\dfrac{x(r(\mu))}{X\beta^{|\mu|+1}}\\
&\leq \beta' \dfrac{M}{m}
\dfrac{x(r(\mu))}{X\beta^{|\mu|+1}}=
\dfrac{\beta'M}{\beta m}\tau(S_\mu S_\mu^*).
\end{aligned} $$

\end{proof}

\begin{lemma}\label{pf2}
For an irreducible matrix $A$, as above, we set $X=\sum_i x(i)$, $\alpha= \min_i x(i)$, and $\alpha'=\max_i x(i)$. Then 
for every positive integer $k$ we have 
$$
0<\frac{X}{\alpha'} \leq \frac{\sum_{i,j}A^k(i,j)}{\beta^k} \leq \frac{X}{\alpha}. 
$$
\end{lemma}
\begin{proof}
Since $x(j)/\alpha' \leq 1 \leq x(j)/\alpha$ for all $j$ and $\sum_j A^k(i,j)x(j)=\beta^k x(i)$ for all $i$, we have 
$$
\frac{X}{\alpha'} = \frac{\sum_{i,j}A^k(i,j)x(j)}{\beta^k\alpha'}  \leq \frac{\sum_{i,j}A^k(i,j)}{\beta^k} \leq 
\frac{\sum_{i,j}A^k(i,j)x(j)}{\beta^k\alpha} = \frac{X}{\alpha}. 
$$
\end{proof}

For an irreducible matrix $A$, as above, and a fixed pair of indices $(i_1,j_1)$ we set 
\begin{equation}\label{a1}
A_1(i,j) := \left\{ \begin{array}{ll} A(i,j) & \text{if } (i,j) \neq (i_1,j_1) \\ 1 & \text{if } (i,j) = (i_1,j_1) \end{array}  \right. 
\end{equation}

\begin{theorem}\label{pf3}
Let $A$ be an irreducible matrix, as above. Assume that $A(i_1,j_1)\geq 2$. Then $A_1$ is an irreducible matrix such that $A_1\leq A$ 
and we have 
$$
\lim_{k\to\infty} \frac{\sum_{i,j}A_1^k(i,j)}{\beta^k} =0, 
$$
with $\beta$ the Perron-Frobenius eigenvalue of $A$.
\end{theorem}
\begin{proof}
It is clear that $A_1$ is irreducible and $A_1\leq A$. Let $\beta_1$ be the Perron-Frobenius eigenvalue of $A_1$, with the corresponding 
Perron-Frobenius eigenvector $(x_1(1),\ldots,x_1(n))$. We have 
$$
\frac{\sum_{i,j}A_1^k(i,j)}{\beta^k} = \frac{\sum_{i,j}A_1^k(i,j)}{\beta_1^k}\cdot\frac{\beta_1^k}{\beta^k}. 
$$
Thus, in view of Lemma \ref{pf2}, it suffices to show that $\beta_1<\beta$.

Now, for each pair of indices $(i,j)$ we can find an $l_{i,j}$ such that $A_1^{l_{i,j}}(i,j) < A^{l_{i,j}}(i,j)$. Indeed, denote by $E_1$ a graph with 
the adjacency matrix  $A_1$. We may view $E_1$ as a subgraph of $E$. Given $(i,j)$ we can find a path $\mu\in E^*\setminus E_1^*$ with 
source in vertex $i$ and range in vertex $j$. To this end take a path $\mu_1$ from $i$ to $i_1$, a path  $\mu_2$ from $j_1$ to $j$, and an 
edge $e\in E^1\setminus E_1^1$ from $i_1$ to $j_1$. Then put $\mu:=\mu_1 e\mu_2$. Setting $l_{i,j}:=|\mu|$ we have $A_1^{l_{i,j}}(i,j) < 
A^{l_{i,j}}(i,j)$, as desired. Let $k$ be an integer such that $k > l_{i,j}$ for all $i,j$. Then we have 
$$
\sum_{j=1}^k A_1^j < \sum_{j=1}^k A^j .
$$
Now, we set $\overline{A}=\sum_{j=1}^k A^j$, $\overline{A}_1=\sum_{j=1}^k A^j_1$, $\overline{\beta}=\sum_{j=1}^k \beta^j$, and 
$\overline{\beta}_1=\sum_{j=1}^k \beta^j_1$. We have $\overline{A}x=\overline{\beta}x$ and $\overline{A}_1x_1=\overline{\beta}_1x_1$. 
To prove the theorem, it suffices to show that $\overline{\beta}_1 < \overline{\beta}$. Thus without loss of generality we may simply assume that 
$A_1 < A$.

Let $I$ be the $n\times n$ matrix with $I(i,j)=1$ for all $i,j$. Since $A > A_1$, we have 
$$
A \geq A_1 + I. 
$$
With $X_1:=\sum_j x_1(j) >0$, we see that 
$$
Ax_1 \geq (A_1+I)x_1 = \beta_1x_1 + \left(  \begin{array}{c} X_1 \\ \vdots \\ X_1    \end{array}   \right). 
$$
We can take a small $\epsilon>0$ such that 
$$
\beta_1x_1 + \left(  \begin{array}{c} X_1 \\ \vdots \\ X_1    \end{array}   \right) \geq (\beta_1+\epsilon)x_1
$$
This means that $\lambda(x_1,A) \geq \beta_1+\epsilon > \beta_1$. Since $\beta \geq \lambda(x_1,A)$, we may finally conclude that $\beta>\beta_1$. 
\end{proof}


\vspace{2mm}
\medskip\noindent
Tomohiro Hayashi\\
Nagoya Institute of Technology\\
Gokiso-cho, Showa-ku, Nagoya, Aichi\\
466--8555, Japan\\
email: hayashi.tomohiro@nitech.ac.jp\\

\smallskip\noindent
Jeong Hee Hong\\
Department of Mathematics and Computer Science\\
The University of Southern Denmark\\
Campusvej 55, DK--5230 Odense M\\
Denmark\\
email: hongjh@imada.sdu.dk\\

\smallskip\noindent
Wojciech Szyma{\'n}ski\\
Department of Mathematics and Computer Science\\
The University of Southern Denmark\\
Campusvej 55, DK--5230 Odense M\\
Denmark\\
email: szymanski@imada.sdu.dk

\vspace{2mm}
\bigskip\noindent
{\bf Competing interests:} The authors declare none.


\begin{thebibliography}{20} 

\bibitem{AJSz} J. E. Avery, R. Johansen and W. Szyma\'{n}ski, 
{\it Visualizing automorphisms of graph algebras}, 
Proc. Edinburgh Math. Soc. {\bf 61} (2018), 215--249. 

\bibitem{BP} T. Bates and D. Pask, 
{\it Flow equivalence of graph algebras}, 
Ergod. Th. \& Dynam. Sys. {\bf 24} (2004), 367--382.

\bibitem{BPRSz} T. Bates, D. Pask, I. Raeburn and W. Szyma\'{n}ski, 
{\it The $C^*$-algebras of row-finite graphs},  
New York J. Math. {\bf 6} (2000), 307--324. 

\bibitem{CHS2012} R. Conti, J. H. Hong and W. Szyma{\'n}ski,
{\em Endomorphisms of graph algebras,} 
J. Funct. Anal. {\bf 263} (2012), 2529--2554. 

\bibitem{CHS2015} R. Conti, J. H. Hong and W. Szyma{\'n}ski,
{\it On conjugacy of MASAs and the outer automorphism group of the Cuntz algebra,} 
Proc. Royal Soc. Edinburgh {\bf 145} (2015), 269--279. 

\bibitem{CP}  R. Conti and C. Pinzari, 
{\it Remarks on the index of endomorphisms of Cuntz algebras}, 
J. Funct. Anal. {\bf 142} (1996), 369--405. 

\bibitem{CSz}  R. Conti and W. Szyma{\'n}ski, 
{\it Labeled trees and localized automorphisms of the Cuntz algebras,} 
Trans. Amer. Math. Soc. {\bf 363} (2011), 5847--5870. 

\bibitem{Cun77}  J. Cuntz,
{\it Simple $C^*$-algebras generated by isometries}, 
Comm. Math. Phys. {\bf 57} (1977), 173--185. 

\bibitem{Cun}  J. Cuntz,
{\it Automorphisms of certain simple $C^*$-algebras},
in {\it Quantum fields-algebras-processes} (Bielefield, 1978), pp. 187--196, ed. 
L. Streit, Springer, Vienna, 1980.

\bibitem{CK} J. Cuntz and W. Krieger, 
{\it A class of $C^*$-algebras and topological Markov chains}, 
Invent. Math. {\bf 56}  (1980), 251--268.

\bibitem{DP} A. P. Donsig and S. C. Power, 
{\it The failure of approximate inner conjugacy for standard diagonals in regular limit algebras}, 
Canad. Math. Bull. {\bf 39} (1996), 420--428. 

\bibitem{GHJ} F. M. Goodman, P. de la Harpe and V. F. R. Jones, 
{\it Coxeter graphs and towers of algebras}, 
M. S. R. I. Publ., 14. Springer-Verlag, New York, 1989.

\bibitem{G} A. Graham, 
{\it Nonnegative matrices and applicable topics in Linear Algebra}, 
John Wiley \& Sons, New York, 1987. 

\bibitem{HHMSz} T. Hayashi, J. H. Hong, S. E. Mikkelsen and W. Szyma\'{n}ski, 
{\it On conjugacy of subalgebras of  graph $C^*$-algebras,} 
Proceedings of the 37$^{\rm th}$ Workshop on Geometric Methods in Physics (Bia{\l}owie\.{z}a, July 2018),   163--173,   
Trends in Mathematics, Birkh\"{a}user, 2019. 

\bibitem{HPP} A. Hopenwasser, J. R. Peters and S. C. Power, 
{\it Subalgebras of graph $C^*$-algebras}, 
New York J. Math. {\bf 11} (2005), 351--386. 

\bibitem{HoIs} C. Houdayer and Y. Isono, 
{\it Unique prime factorization and bicentralizer problem for a class of type $III$ factors}, 
Adv. Math.  {\bf 305}  (2017), 402--455.

\bibitem{HoVa} C. Houdayer and S. Vaes, 
{\it Type $III$ factors with unique Cartan decomposition}, 
J. Math. Pures Appl. {\bf 100} (2013), 564--590. 

\bibitem{JSSz} R. Johansen, A. P. W. S{\o}rensen and W. Szyma\'{n}ski, 
{\it The polynomial endomorphisms of graph algebras,} 
Groups, Geometry, and Dynamics {\bf 14} (2020), 1043--1075.  

\bibitem{KPR} A. Kumjian, D. Pask and I. Raeburn, 
{\em Cuntz-Krieger algebras of directed graphs}, 
Pacific J. Math. {\bf 184} (1998), 161--174.

\bibitem{KPRR} A. Kumjian, D. Pask, I. Raeburn and J. Renault, 
{\em Graphs, groupoids, and Cuntz-Krieger algebras}, 
J. Funct. Anal. {\bf 144} (1997), 505--541. 

\bibitem{Packer} J. A. Packer, 
{\it Point spectrum of ergodic abelian group actions and the corresponding group-measure factors}, 
Pacific J. Math. {\bf 119} (1985), 381--d405. 

\bibitem{Popa}  S. Popa, 
{\it Strong rigidity of type $II_1$ factors arising from malleable actions of $w$-rigid groups, I}, 
Invent. Math. {\bf 165} (2006), 369--408. 

\bibitem{SpV} A. Speelman and S. Vaes, 
{\it A class of $II_1$ factors with many non conjugate Cartan subalgebras}, 
Adv. Math. {\bf 231} (2012), 2224--2251. 

\bibitem{Z} J. Zacharias, 
{\it Quasi-free automorphisms of Cuntz-Krieger-Pimsner algebras,} 
in {\em $C^*$-algebras} (M\"{u}nster, 1999), 262--272, Springer, Berlin, 2000. 

\end{thebibliography}
\end{document}